\documentclass{amsart}

\title{On Strassen's Theorem for support functions}
\author{Stefan Schrott and Daniel Toneian} 
\date{August 9, 2024}

\usepackage[english]{babel}
\usepackage[utf8]{inputenc}
\usepackage[T1]{fontenc}
\usepackage[centering]{geometry}

\usepackage{amsmath}
\usepackage{amssymb}
\usepackage[dvipsnames]{xcolor}

\usepackage{enumerate}

\usepackage{geometry}
\usepackage{scalerel} 

\usepackage{mathtools}

\usepackage{amsthm}
\newtheorem{theorem}{Theorem}
\numberwithin{theorem}{section}
\newtheorem{lemma}[theorem]{Lemma}
\newtheorem{proposition}[theorem]{Proposition}
\newtheorem{corollary}[theorem]{Corollary}
\theoremstyle{definition}
\newtheorem{definition}[theorem]{Definition}

\newtheorem{remark}[theorem]{Remark}

\let\oldmarginpar\marginpar
\renewcommand\marginpar[1]{\-\oldmarginpar[\raggedleft\footnotesize #1]{\raggedright\footnotesize\color{red} #1}}
\renewcommand{\epsilon}{\varepsilon}
\renewcommand{\subset}{\subseteq}

\newcommand{\N}{\mathbb{N}}

\newcommand{\R}{\mathbb{R}}

\newcommand{\Sp}{\mathbb{S}}

\newcommand{\prob}{\mathcal{P}}

\newcommand{\F}{\mathcal{F}}

\newcommand{\M}{\mathcal{M}}
\newcommand{\W}{\mathcal{W}}

\newcommand{\pr}{\textup{pr}}
\newcommand{\cont}{\textup{cont}}
\newcommand{\dom}{\textup{dom}}

\newcommand{\phc}{\preceq_{\mathrm{s}}}
\newcommand{\cx}{\preceq_{\mathrm{c}}}
\newcommand{\eqph}{\equiv_{\mathrm{ph}}}
\newcommand{\ba}{\textup{{m}}}
\newcommand{\Qm}{\mathcal{Q}_\mathrm{m}}

\subjclass[2020]{60A10, 60G48, 49J55}
\keywords{Strassen's theorem, support functions, convex order}

\begin{document}

\maketitle

\begin{abstract}
Strassen \cite{St65} established that there exists a two step martingale with marginal distributions $\mu$, $\nu$ if and only if  $\mu$, $\nu$ are in convex order. 
	Recently Choné, Gozlan and Kramarz  \cite{ChGoKr23} obtained a transport characterization of the stochastic order defined by convex positively 1-homogeneous functions, in the spirit of Strassen’s theorem under certain technical assumptions. 
	In this note we prove the result of \cite{ChGoKr23} in full generality. We also observe that the restriction of the result to the case where $\mu, \nu$ are supported on a half space is equivalent to Strassen’s classical theorem.
\end{abstract}

\section{Introduction}
A martingale coupling between two probability measures $\mu,\nu \in \prob(\R^d)$ is a probability measure $\pi \in \prob(\R^d \times \R^d)$ with first marginal $\mu$ and second marginal $\nu$ such that the canonical process is a two-step martingale under $\pi$, or equivalently,  $x = \int y \, \pi^x(dy)$ $\mu$-a.s. for a disintegration $(\pi^x)_{x \in \R^d}$ of $\pi$ w.r.t.\ the first marginal. The probability measures $\mu,\nu \in \prob(\R^d)$ are said to be in convex order, denoted by $\mu \cx \nu$, if $\int f d\mu \le \int f d\nu$ for all convex $f : \R^d \to \R$. Strassen’s celebrated theorem connects these two concepts:
\begin{theorem}[Strassen's theorem]\label{thm:strassen}
	There is a martingale coupling between $\mu,\nu \in \prob(\R^d)$ if and only if $\mu \cx \nu$. 
\end{theorem} 



{In \cite{ChGoKr23}, Choné, Gozlan and Kramarz } consider a weak optimal transport problem for unnormalized kernels (see Definition~\ref{def:kernel} below). They establish a Kantorovich-type duality for this transport problem and as an application of this duality derive a Strassen-type theorem for the order on the set of positive measures with finite first moments $\M_1(\R^d)$ induced by testing against support functions.

\begin{definition}
	A function $f: \R^d \to \R$ is a support function if it is convex and positively 1\hbox{-}homogeneous, i.e.\  $f(t x) = t f(x)$ for all $x \in \R^d$ and $t\ge 0$. We write $\mu \phc \nu$ if $ \int f \, d\mu \le \int f \, d\nu$   for every support function  $f: \R^d \to \R$.
\end{definition}
Due to the one-to-one correspondence between support functions and convex bodies, support functions are a central tool in convex geometry. 
\begin{definition}\label{def:kernel}
	An unnormalized kernel $q= (q^x)_{x \in \R^d}$ is a collection of finite non-negative  measures such that the map $ x \mapsto q^x(A)$ is Borel measurable for every Borel set $A \subset \R^d$.
	
	A kernel $q$ transports $\mu$ to $\nu$ if  
	$$
	\int  \, q^x(dy)  \, \mu(dx)=   \nu(dy).
	$$
	The set of all kernels that transport $\mu$ to $\nu$ is denoted by $\mathcal{Q}(\mu,\nu)$.
	
	A kernel $q \in \mathcal{Q}(\mu,\nu)$ is called moment-preserving if $$\int y\, q^x(dy)=x$$for $\mu$-a.e.\ $x$. The collection of moment-preserving kernels from $\mu$ to $\nu$ is denoted by $\Qm(\mu,\nu)$. 
\end{definition}

{Choné, Gozlan and Kramarz} \cite[Theorem 5.2]{ChGoKr23} establish the following result:

\begin{theorem}\label{thm:Strassen_Gozlan}
	Let $\mu,\nu \in \M_1(\R^d)$ and suppose that $\nu$ is compactly supported.
	\begin{enumerate}[(a)]
		\item If $\mu$ is compactly supported as well then the following are equivalent:
		\begin{enumerate}[(i)]
			\item $\mu \phc \nu$
			\item $\Qm(\mu,\nu) \neq \emptyset$
		\end{enumerate} 
		\item The same conclusion holds if the convex hull of the support of {$\nu$} does not contain 0.
	\end{enumerate}
\end{theorem}
The main result of this note is to establish this in full generality: 

\begin{theorem}\label{thm:main}
	Let $\mu, \nu \in \M_1(\R^d)$. Then the following are equivalent:
	\begin{enumerate}[(i)]
		\item $\mu \phc \nu$ 
		\item $\Qm(\mu,\nu) \neq \emptyset$
	\end{enumerate}
\end{theorem}
In contrast to the proof given {in} \cite{ChGoKr23}, our proof of Theorem~\ref{thm:main} does not use transport duality, but is based on a combination of explicit geometric constructions and functional analytic arguments which are inspired by Strassen \cite{St65}. 

\subsection*{Related literature}
The convex order and 
Strassen's theorem play an important role in different developments in probability, e.g.\ the existence of continuous (Markov-) martingales with given marginals \cite{Ke72, Lo08b, PaRoSc22}, the PCOC-problem \cite{HiYo10, HiPrRoYo11}, the Skorokhod embedding problem, which  has a solution if and only if initial and terminal laws are in convex order \cite{Ob04, BeCoHu14}, and the martingale transport problem   \cite{HoNe12, BeHePe12, GaHeTo13, BeNuTo16}.  

Weak optimal transport \cite{GoRoSaTe14} as well as its version for unnormalized kernels \cite{ChGoKr23} present extensions of classical transport (see \cite{Vi09, Sa15, FiGl21} for recent monographs) which capture several applications which are beyond the usual framework, see \cite{BaPa19} for an overview. Also, the notion of barycentric costs considered in Section~\ref{sec:general_measures} below arose in the context of weak transport theory, see \cite{GoJu18}.

\subsection*{Organisation of the paper and further results}

In Section~2 we present a characterization of the equivalence relation induced by $\phc$ (see Proposition~\ref{prop:char_ph_equiv}). As every class of this relation has a member that is compactly supported, these considerations imply that it suffices to show  Theorem~\ref{thm:main} for compactly supported measures. 

In Section~3 we give a new proof of Theorem~\ref{thm:Strassen_Gozlan}(a), which is inspired by the approach used in Strassen's original paper \cite{St65}. We note that the results of Section~2 together with the proof of  Theorem~\ref{thm:Strassen_Gozlan}(a) given in \cite{ChGoKr23} already suffice to prove Theorem~\ref{thm:main}.

In Section~4 we note the equivalence of Theorem~\ref{thm:strassen} and Theorem~\ref{thm:Strassen_Gozlan}(b) in the sense that one can be proven from the other by simple projection/embedding-arguments. In this sense Theorem~\ref{thm:main}  is indeed a generalization of Strassen's theorem.

\section{Properties of $\phc$}\label{sec:2}

The relation $\phc$ is not a partial order, but only a preorder as it lacks antisymmetry. Hence, it is natural to consider the equivalence relation which is induced by this preorder, i.e.\  $\mu$ and $\nu$ are equivalent if $\mu \phc \nu$ and $\nu \phc  \mu$. It turns out that this equivalence relation is given by $\eqph$ (see Definition~\ref{def:pheq}) and that a `symmetrized version' of Theorem~\ref{thm:main} can be proven much easier than  Theorem~\ref{thm:main} itself (see Proposition~\ref{prop:char_ph_equiv}). From this it follows quickly that it suffices to prove Theorem~\ref{thm:main} for measures that are concentrated on the unit-sphere $\Sp^{d-1}$, showing that compactness assumptions on the supports of the measures (as in Theorem~\ref{thm:Strassen_Gozlan}) are not needed. 


If a positively 1-homogeneous function $f$ is bounded on the unit sphere $\Sp^{d-1}$, there is a constant $c$ such that $|f(x)| \le c|x|$ for all $x \in \R^d$, so integration against measures in $\mathcal{M}_1(\R^d)$ is well-defined. As $\Sp^{d-1}$ is compact, the same is true for all continuous positively 1-homogeneous functions and in particular for all {real valued} support functions. 

{
We will sometimes allow convex functions to attain the value $+ \infty$ and always indicate such instances explicitly. As such functions are bounded from below by an affine function, integration against measures in $\M_1(\R^d)$ is still well-defined in the sense that the integral can be $+\infty$.} 
{
\begin{remark}\label{remark:extended_values}
	 Given $\mu \phc \nu \in \M_1(\R^d)$ and a lower semi-continuous (l.s.c.) support function $f: \R^d \to \R \cup \{+\infty\}$, we have $\int f \, d\mu \leq \int f \, d\nu$ (where both sides of the inequality could be  $+\infty$). This follows from monotone convergence because every l.s.c.\ support function $f: \R^d \to \R \cup \{+ \infty\}$ is an increasing limit of finite maxima of linear functions.
\end{remark}
}

\begin{definition}\label{def:pheq}
	Let $\mu, \nu \in \mathcal{M}_1(\R^d)$. We write $\mu \eqph \nu$ {for the equivalence relation induced by  positively 1-homogeneous functions, i.e.\  $\mu \eqph \nu$ if $\int f \,d\mu = \int f \,d\nu$} for every positively 1-homogeneous Borel function $f:\R^d \to \R$ that is bounded on $\Sp^{d-1}$.
\end{definition}
The following standard fact from convex geometry can be found in \cite[Lemma~1.7.8]{Schneider1993}. It allows us to easily identify $\eqph$ as the equivalence relation induced by the preorder $\phc$. 
\begin{lemma}\label{lemma:C2_supfct}
	Let $f \in C^2(\R^d \setminus \{0\})$ be positively 1-homogeneous. Then there exists a constant $c>0$ so that $g(x):= f(x) +c|x|$ is a support function.
\end{lemma}
\begin{corollary}\label{cor:ph=phc}
	Let $\mu,\nu \in \M_1(\R^d)$. Then $\mu \eqph \nu$ if and only if $\mu \phc \nu$ and $\nu \phc \mu$.
\end{corollary}
\begin{proof}
	Assume that $\mu \phc \nu$ and $\nu \phc \mu$. We then have $\int f \, d\mu = \int f \, d\nu$ for every support function $f: \R^d \to \R$. By Lemma~\ref{lemma:C2_supfct} every positively 1-homogeneous $g \in C_2(\R^d \setminus  \{0\})$ is a difference of two support functions, so we have $\int g \,d\mu = \int g \,d \nu$ as well. By standard approximation arguments, the equality carries over to all positively 1-homogeneous functions, thus  $\mu \eqph \nu$. The converse implication is trivial.
\end{proof}
There is a one-to-one correspondence between positively 1-homogeneous functions on $\R^d$ and functions on the unit sphere $\Sp^{d-1}$: Given a function $f: \Sp^{d-1} \to \R$ we can extend it in a unique way to a positively 1-homogeneous function $\bar f : \R^d \to \R$ by setting
\begin{align}\label{eq:homogen_ext}
	\bar f(x) = \begin{cases} 
		|x| f\big(  \frac{x}{|x|} \big)  & x \neq 0,\\
		0 & x=0.
	\end{cases}
\end{align}
Then the restriction of $\bar f$ to $\Sp^{d-1}$ equals $f$. {Using this {one-to-one correspondence}, it is common to consider support functions either as functions on $\R^d$ or as functions on $\mathbb{S}^{d-1}$, and both notions are used interchangeably.}

{The observation above implies} that every $\eqph$-class contains exactly one canonical representative that is concentrated on the unit sphere:

\begin{definition}\label{def:homogmarg}
	Let $\mu \in \M_1(\R^d)$. We define its homogeneous marginal  $\mu_\Sp \in \M_1(\Sp^{d-1})$ via
	$$
	\int_{\Sp^{d-1}} f(u) \, \mu_\Sp(du) = \int_{\R^d} \bar f(x)  \,\mu(dx) \qquad f : \Sp^{d-1} \to \R, \text{ Borel}
	$$
	where $\bar f$ is defined as in \eqref{eq:homogen_ext}.
\end{definition}
{The total mass of $\mu_\Sp$  equals the first absolute moment of $\mu$, in particular, $\mu_\Sp$ is a finite measure.} Clearly, $\mu \eqph \mu_\Sp$, i.e.\  $\mu_\Sp$ is indeed a representative of the $\eqph$-class of $\mu$. 

In order to keep  notation short, we {abbreviate the first moment of a measure $\mu$ by $\ba(\mu)$, i.e.}
\begin{align}
	\ba(\mu):= \int x\, \mu(dx)  .
\end{align}

\begin{proposition}\label{prop:sphere}
	Let $\mu \in \M_1(\R^d)$, $\mu \neq c \delta_0$ for every $c >0$. Then $\Qm(\mu,\mu_\Sp)$ and $\Qm(\mu_\Sp,\mu)$ are both non-empty. 
\end{proposition}

\begin{proof}
	It is easy to check that the kernel $p$ defined by $p^x:=|x| \, \delta_{x/|x|}$ for $x\neq 0$ and $p^0:=0$ is in $\Qm(\mu,\mu_\Sp)$. Next, we construct $q \in \Qm(\mu_\Sp,\mu)$. 
	{To do this, consider the map $u(x):=\frac{x}{|x|}$ and let $\bar \mu := u_\#\mu = \mu \circ u^{-1} \in \mathcal{M}(\Sp^{d-1})$ be the pushforward of $\mu$ under the map $u$. By the disintegration theorem  (see e.g.\ \cite[Exercise~17.35]{Ke95}) there is a kernel $(\mu^u)_{u \in \Sp^{d-1}}$ such that $\mu^u \in \prob(\R^d)$ is supported on the ray $\{tu : t > 0\}$ for $\bar \mu$-a.e.\ $u$ and such that for every Borel $f: \R^d \to \R$   } 
	\begin{align}\label{eq:disintprf}
		\int_{\R^d \setminus \{0 \}} f(x) \, \mu(dx) = \iint  f(x) \, \mu^{u}(dx) \bar{\mu}(du).   
	\end{align}
	Using this we define $q$ by 
	$$q^u:=\frac{1}{|\ba(\mu^{u})|}\mu^{u}+\frac{\mu(\{0\})}{\mu_\Sp(\Sp^{d-1})} \delta_0 
	$$
	for $u \in \Sp^{d-1}$. It is easy to see that $\ba(q^u)=u$. For every Borel $f : \Sp^{d-1} \to \R$ we have 
	$$
	\int f(u) |\ba(\mu^u)| \,\bar \mu(du)= \int f(u) \int |y|\, \mu^u(dy) \bar{\mu}(du) = \int f(u(x)) |x|\, \mu(dx) = \int f(u)\, \mu_\Sp(du),
	$$
	where the last equality is just the definition of $\mu_\Sp$. Hence, $\bar\mu(du)= \frac{1}{|\ba(\mu^u)|} \mu_\Sp(du)$. Using this and \eqref{eq:disintprf} we find for every bounded Borel $g: \R^d \to \R$ 
	\[
	\int g(x) \, q^u(dx) \mu_\Sp(du)= \iint g(x)  \mu^u(dx) \frac{1}{|\ba(\mu^u)|}\mu_\Sp(du)   + g(0)\mu({\{ }0{\}}) = \int g(x) \, \mu(dx). \qedhere
	\]

\end{proof}

\begin{remark}
	Note that $\mu_\Sp = 0$ if and only if $\mu = c \delta_0$ for some $c>0$. As it is not possible to transport the 0 measure to any non-zero measure, our constructions fail for such $\mu$. However, Theorem~\ref{thm:main} is still valid in that case and trivial to prove: {Assume that $\mu=c\delta_0 \phc \nu$. By testing against the linear functions $x \mapsto x_i$ and $x \mapsto -x_i$ for $1 \le i \le d$, we obtain that  $\int y \,\nu(dy)=0$. Hence,}   the kernel defined via $p^0=\frac{1}{c}\nu$ is in $\Qm(\mu,\nu)$. In order to avoid tedious case distinctions, we will exclude the special case $\mu = c \delta_0$ further on.
\end{remark}

\begin{lemma}[Gluing] \label{lemma:gluing}
	Let $\mu,\nu,\rho \in \M_1(\R^d)$ and kernels 
	$p \in \mathcal{Q}(\mu,\nu)$ and $q \in \mathcal{Q}(\nu,\rho)$ {be} given. Then the kernel $r$ defined by $r^x= \int q^y p^x(dy)$ is in  $\mathcal{Q}(\mu,\rho)$ and satisfies
	\begin{align}\label{eq:gluingineq}
		\int \left| x-  \ba(r^x) \right| \, \mu(dx) \leq \int \left| x- \ba(p^x)  \right| \, \mu(dx) + \int \left| y- \ba(q^y) \right| \, \nu(dy) .   
	\end{align}   
	In particular, if $p \in \Qm(\mu,\nu)$ and $q \in  \Qm(\nu,\rho)$, we have $r\in  \Qm(\mu,\rho)$ as well.
\end{lemma}

\begin{proof}
{It is easy to check that $r$ transports $\mu$ to $\rho$.}
	In order to show \eqref{eq:gluingineq}, we calculate
	\begin{align*}
		\int \left| x- \ba(r^x)  \right|  \, \mu(dx)
		&= \int \left| x- \int \ba(q^y) p^x(dy) \right| \, \mu(dx) 
		\\
		&= \int \left| x- \ba(p^x) + \ba(p^x) - \int \ba(q^y) p^x(dy) \right| \, \mu(dx) \\
		&\leq \int \left| x- \ba(p^x)\right| \, \mu(dx) + \iint\left| y  - \ba(q^y)  \right| \, p^x(dy) \mu(dx) \\
		&= \int \left| x- \ba(p^x)\right| \, \mu(dx) + \int  \left| y  - \ba(q^y)  \right| \, \nu(dy). \qedhere
	\end{align*}
\end{proof}



\begin{proposition}\label{prop:char_ph_equiv}
	For $\mu, \nu \in \M_1(\R^d)$ the following are equivalent: 
	\begin{enumerate}[(i)]
		\item $\Qm(\mu,\nu) \neq \emptyset$ and $\Qm(\nu,\mu) \neq \emptyset$
		\item $\mu \phc \nu$ and $\nu \phc \mu$
		\item $\mu \eqph \nu$
		\item $\mu_\Sp = \nu_\Sp$
	\end{enumerate}
\end{proposition}
\begin{proof}
	First show that (i) implies (ii). To this end, let us assume the existence of $p \in \Qm(\mu,\nu)$ and let $f$ be a support function. Then, by Jensen's inequality {and the positive 1-homogeneity of $f$}, we find
	\begin{align}\label{eq:jensen}
		\int f(x) \, \mu(dx)= \int f \left(\int y \, p^x(dy)\right) \, \mu(dx) \leq \iint f(y) \, p^x(dy) \, \mu(dx) = \int f(y) \, \nu(dy).
	\end{align}
	The reverse inequality is proved the same way. The equivalence of (ii) and (iii) is precisely Corollary~\ref{cor:ph=phc}. 	 {Further, (iii) implies (iv) because we have for every Borel $f: \Sp^{d-1} \to \R$}
	\[
	\int f(u) \, \mu_\Sp(du)= \int \bar{f}(x) \, \mu(dx) = \int \bar{f}(x) \, \nu(dx)= \int f(u) \, \nu_\Sp.
	\]
	To conclude the proof, we show that (iv) implies (i). {Indeed, } Proposition~\ref{prop:sphere} guarantees that $\mathcal{Q}(\mu,\mu_\Sp)\neq \emptyset$ and $\mathcal{Q}(\nu_\Sp, \nu)\neq \emptyset$ and thus  $\mathcal{Q}(\mu,\nu)\neq \emptyset$ {by Lemma~\ref{lemma:gluing}.}  
\end{proof}
If $\mu, \nu \in \prob(\R^d)$ {with finite second moment} are in convex order and satisfy $\int \! |x|^2 \mu(dx) = \int |x|^2 \,\nu(dx)$, then $\mu=\nu$. 
There is an analogous phenomenon for the $\phc$-order: 

\begin{remark}
	Let $\mu, \nu \in \M_1(\R^d)$ satisfy $\mu \phc \nu$ and  $\int |x| \, \mu(dx) = \int |x| \, \nu(dx)$. Then $\mu \eqph \nu$.  To see this, let $f \in C^2(\R^d \setminus \{0\})$ be positively 1-homogeneous. Applying Lemma~\ref{lemma:C2_supfct} to $f$ and $-f$ yields the existence of support functions $g_1, g_2$ and constants $c_1,c_2$ such that $f(x)=g_1(x)-c_1|x|= -g_2(x) + c_2|x|$. This implies $\int f \,d\mu = \int f\, d\nu$. By standard approximation results, we conclude $\int f \,d\mu = \int f \, d\nu$ for every 
	Borel {positively 1}-homogeneous $f : \R^d \to \R$. 
\end{remark}

{We close this section with the following observation.
\begin{remark}\label{rem:cospt}
Write $\mathrm{co}(\mathrm{supp}(\mu))$ for the closed convex cone generated	by the support of $\mu$. Then $\mathrm{co}(\mathrm{supp}(\mu))$ is the intersection over all closed half-spaces $\{x \in \R^d: x \cdot u \le 0\}$ such that $\int (x \cdot u)_+ \, \mu(dx) =0$. Hence, if $\mu \phc \nu$ we have $\mathrm{co}(\mathrm{supp}(\mu)) \subseteq \mathrm{co}(\mathrm{supp}(\nu))$.
\end{remark}}


\section{Functional analytic proof of Theorem~\ref{thm:main}}
 
  This section is devoted to the proof of Theorem~\ref{thm:main}. Our proof is a modification of the classical proof of Strassen's theorem given in \cite{St65}. A main step in his proof is the following integral representation for {continuous real-valued} support functions {(i.e.\ convex and positively 1-homogeneous functions) on a Banach space}. 


\begin{proposition}[{\cite[Theorem 1]{St65}}]\label{prop:StrassenThm1}
	Let $(\Omega,\F,{\mu})$ be a probability  space, let $X$ be a separable Banach space, and for each $\omega \in \Omega$ let $h_\omega$ be a continuous support function on $X$ such that for every $x \in X$ the map $\omega \mapsto h_\omega(x)$ is $\F$-measurable. Then $h := \int h_\omega \, \mu(d\omega)$ is a continuous support function on $X$ and for all $\phi$ in the topological dual space $X^\ast$ the following are equivalent:
	\begin{enumerate}[(i)]
		\item $\phi \le h$
		\item There are $\phi_\omega \in X^\ast$ such that $\phi_\omega \le h_\omega$ and $\phi = \int \phi_\omega\, \mu(d\omega)$ such that the map $\omega \mapsto \phi_\omega(x)$ is $\F$-measurable for every $x \in X$.
	\end{enumerate}
\end{proposition}

In order to prove Theorem~\ref{thm:strassen}, Strassen applied Proposition~\ref{prop:StrassenThm1} to the measure space $(\Omega,\mu)$, where $\Omega \subset \R^d$ is a compact {convex} set, and the separable Banach space $X=C(\Omega)$, whose dual is the space of signed Borel measures $\M(\Omega)$ by {the} Riesz representation theorem. For $\omega \in \Omega$ he considered the {following} support function {on C$(\Omega)$}
\begin{align}\label{eq:supp_func_Strassen}
	h_\omega(f):= \inf \{ g(\omega): {g \in C(\Omega)  \text{ concave and } g \ge f } \}.
\end{align}
Let the function $h$ be defined by $h:= \int h_\omega \, \mu(d\omega)$. It turns out that $\mu \cx \nu$  implies that $\nu$ regarded as a linear functional on $C(\Omega)$ is dominated by the support function $h$, i.e.\   $\int f \, d\nu \le h(f)$ for all $f \in C(\Omega)$. Then the family of linear functionals $(\phi_\omega)_{\omega \in \Omega}$ provided by Proposition~\ref{prop:StrassenThm1} can be seen as a kernel by the Riesz representation theorem. It turns out that this kernel is a martingale transport from $\mu$ to $\nu$.

In order to prove Theorem~\ref{thm:main}, we need to replace the support function \eqref{eq:supp_func_Strassen}  by
\begin{align}\label{eq:supp_func_new}
	{h_u(f):= \inf \{ g(u): g \in C(\Sp^{d-1}) \text{ such that $-g$ is a support function and } g \ge f \}}
\end{align}
{for $u \in \Sp^{d-1}$. Note that we replaced the compact convex set $\Omega$ with $\Sp^{d-1}$. This is a natural choice  because $\Sp^{d-1}$ is compact (hence the Riesz representation theorem is available), support functions are already determined by their behavior on $\Sp^{d-1}$, and the constructions in Section~\ref{sec:2} allow us to only consider measures on $\Sp^{d-1}$ for proving Theorem~\ref{thm:main}.} 
However, {the} support function {\eqref{eq:supp_func_new}} is not continuous and it may even take the value $+\infty$. For instance, {$h_u(1)=+\infty$} because the infimum in \eqref{eq:supp_func_new} is empty. The aim of the next {sub}section is to generalize Proposition~\ref{prop:StrassenThm1} in order to fit our needs.

\subsection{Tools from convex analysis}\label{sec:conv_ana}

 Let $X$ be a Banach space and $f : X \to {\R \cup \{+\infty\}}$ be a convex function. Recall that the domain of $f$ is defined as $\dom(f) := \{ x \in X : f(x)< +\infty\}$. We denote with  $\cont(f)$ the set of points in the interior of $\dom(f)$ at which $f$ is continuous. 

The subdifferential of $f$ at $x \in \dom(f)$ is defined as 
$$
\partial f(x) = \{  \phi \in X^\ast : f(y) \ge f(x) + \phi(y-x) \text{ for all } y \in X \}.
$$
A crucial observation is that if $h$ is a support function, we have $\partial h(0) = \{ \phi \in X^\ast : \phi \le h\}$, so in the case that $\mu$ has finite support Proposition~\ref{prop:StrassenThm1} is just an instance of the subdifferential sum rule, which can be found e.g.\ in \cite[Theorem~4.1.19]{BoVa10}{.}
\begin{proposition}\label{prop:subdiff_sum}
	Let $f, g : X \to {\R \cup \{+\infty\}}$ be convex and assume that $\dom(f) \cap \cont(g) \neq \emptyset$. Then we have for all $x \in \dom(f) \cap \dom(g)$
	$$
	\partial (f+g) (x) = \partial f(x) + \partial g (x),
	$$
	where the $+$ on the right hand side is a Minkowski sum.
\end{proposition}
It is easy the see that $\partial (\lambda f)(x) = \lambda \partial f(x)$ for every positive scalar $\lambda$. Therefore we have 

\begin{corollary}\label{cor:replaceStrassen1}
	Let $h_1,\dots, h_n : X \to {\R \cup \{+\infty\}}$ be support functions and $\lambda_1, \dots, \lambda_n >0$. Assume that $\bigcap_{i=1}^n \cont(h_i) \neq \emptyset$. Then $h := \sum_{i=1}^n \lambda_i h_i$ is a support function and for all $\phi \in X^\ast$ the following are equivalent:
	\begin{enumerate}[(i)]
		\item $\phi \le h$
		\item There are $\phi_1,\dots, \phi_n \in X^\ast$ such that $\phi_i \le h_i$ and $\phi= \sum_{i=1}^n \lambda_i \phi_i$.
	\end{enumerate}
\end{corollary}
\begin{proof}
	{Note that since $\cont(h_i) \subseteq \dom(h_i)$ for $1\leq i \leq n$ by definition, all} the requisites for the subdifferential {sum} rule are met {and} we apply the subdifferential {sum} rule to $h$ at the point $x=0$ and find
	
	\[
	\{\phi \in X^* : \phi \leq h\}=\partial h(0)=\sum_{i=1}^n \lambda_i \partial h_i(0). \qedhere
	\]
\end{proof}

\subsection{Proof of Theorem~\ref{thm:main} under the assumption that $\mu$ has finite support}
We adapt Strassen's proof \cite[Theorem~2]{St65} to our setting, using the results from the previous subsection. 

\begin{lemma}\label{lem:proph}
	For {$u  \in \Sp^{d-1}$} consider the function
	\begin{align}\label{eq:supp_func_new2}
		{h_u(f):= \inf \{ g(u): g \in C(\Sp^{d-1}) \text{ such that $-g$ is a support function and } g \ge f \}.}
	\end{align}
	Then we have
	\begin{enumerate}[(i)]
		\item $f\mapsto h_{{u}}(f): C({\Sp^{d-1}}) \to {\R \cup \{+\infty\}} $ is a support function.
		\item ${u} \mapsto -h_{{u}}(f): {\Sp^{d-1}} \to {\R \cup \{+\infty\}} $ is a support function for any $f\in C({\Sp^{d-1}})$.
		\item If $f \le 0$, then $h_{u}(f) \le 0$. In particular, $\dom(h_{u}) \supseteq \{ f \in C({\Sp^d}) : f \le 0 \}$.
		\item The constant function $-1$ is a continuity point of $h_{u}$.
	\end{enumerate}
\end{lemma}
\begin{proof}The items (i) to (iii) are straightforward to check. In order to prove (iv), {note that  $h_{u}(-c)=-c$ for every  $c\ge 0$.} It is easy to see that $f_1 \le f_2$ implies $h_{u}(f_1)\le h_{u}(f_2)$ for all ${u} \in {\Sp^{d-1}}$. Let $\varepsilon \in (0,1)$ and $f \in C({\Sp^{d-1}})$ such that $||-1-f||_\infty < \varepsilon$. Then we have for all ${u} \in {\Sp^{d-1}}$
	$$
	{-}1{-}\varepsilon = h_{u}(-1-\epsilon) \le h_{u}(f) \le h_{u}(-1+\epsilon) = {-}1{+}\epsilon 
	$$
	and hence $|h_{u} (f) - h_{u}(-1)| \le \epsilon$. 
\end{proof}

\begin{remark}
	The functions $h_\omega$ as defined in \eqref{eq:supp_func_Strassen}  which Strassen used to prove Theorem~\ref{thm:strassen} are, up to a sign-convention, just the convex envelope, i.e.\   the function $\omega \mapsto -h_\omega(-f)$ is precisely the convex envelope of $f$. 
	
	At first glance, the function $h_{u}$ that we introduced in \eqref{eq:supp_func_new2} seems to just be a technical modification of that to fit our needs. However, there is a geometric interpretation as well. If $f \in C(\Sp^{d-1})$, $f \ge 0$ the Wulff-shape associated to $f$ is defined by 
	$$
	W_f := \{ x \in \R^d : x \cdot u \le f(u) \text{ for all } u \in \Sp^{d-1} \}.
	$$
	Then the function ${u} \mapsto -h_{{u}}(-{f})$ is precisely the support function of the Wulff-shape $W_{f}$. The Wulff-shape has many applications in convex geometry and crystallography, see for example \cite[Chapter 7.5]{Schneider1993}.
	
\end{remark}

\begin{proposition}\label{prop:mu_endlich}
	Let $\mu \phc \nu \in \M_1({\Sp^{d-1}})$. { If $\mu$ is finitely supported, then $\Qm(\mu,\nu) \neq \emptyset$.}
\end{proposition}
\begin{proof}
	As $\mu$ has finite support, we can write it as $\mu = \sum_{i=1}^n \lambda_i \delta_{{u}_i}$ with $\lambda_i > 0$ and ${u}_i \in {\Sp^{d-1}}$.  We define $h := \sum_{i=1}^n \lambda_i h_{{u}_i} $, where $h_{{u}_i}$ is defined as in~\eqref{eq:supp_func_new2}. As $f({u}) \le h_{u}(f)$ for every ${u} \in {\Sp^d}$ and $f \in C({\Sp^d})$, we have
	$$
	\int f({u})\, \nu(d{u}) \le \int h_{u}(f) \, \nu(d{u}) \le \int h_{u}(f) \,\mu(d{u}) = \sum_{i=1}^n \lambda_i h_{{u}_i}(f) = h(f).
	$$
	So, $\nu$ regarded as a linear functional on $C({\Sp^{d-1}})$ is dominated by the support function $h : C({\Sp^{d-1}}) \to \R$. By Corollary~\ref{cor:replaceStrassen1} and the Riesz-representation theorem, there exist signed measures $
	q^{{u}_i}$ such that 
	\begin{enumerate}[(i)]
		\item $\nu = \sum_{i=1}^n \lambda_i q^{{u}_i} $ and
		\item $\int f\, dq^{{u}_i} \le h_{{u}_i}(f)$ for all $ \text{for all } f \in C({\Sp^{d-1}})$.
	\end{enumerate}
	For every $f \in C({\Sp^{d-1}})$ satisfying $f \le 0$ we have by (ii) and Lemma~\ref{lem:proph}(iii)
	$$
	\int f \,dq^{{u} _i} \le h_{{u}_i} (f) \le 0,
	$$
	so $q^{{u} _i}  $ is a positive measure. Property (i) ensures that the family $q:=(q^{{u}_i})_{1 \leq i \leq n}$, interpreted as a kernel, transports $\mu$ to $\nu$.
	
	Lastly, we need to check that  $q\in \Qm(\mu,\nu)$. {If $f \in C({\Sp^{d-1}})$ is linear}, we have
	$$
	\int f\, dq^{{u}_i} \le h_{{u}_i}(f) = f({u}_i).
	$$
	We apply this to the functions  $f(x)=\pm x_{j}$ for $1\leq {j} \leq d$ and derive $\ba(q^{{u}_i})={u}_i$.
\end{proof}

\subsection{Proof for general measures}\label{sec:general_measures}
{Now that we have shown the main result for $ { \mu \in \M_1(\Sp^{d-1})}$ with finite support, the next step is to drop this restricting condition on $\mu$. To achieve this, {we use} the barycentric cost to quantify how close {two measures} are to having a moment-preserving kernel between them. 

\begin{definition}
	The barycentric cost between  $\mu, \nu \in \mathcal{M}_1(\R^d)$ is defined by
	\[
	\mathrm{BarC}(\mu,\nu) := \inf_{q \in \mathcal{Q}(\mu,\nu)} \int \left|x-\ba(q^x) \right| \, \mu (dx).
	\]
\end{definition}
Note that \eqref{eq:gluingineq} implies that the barycentric cost satisfies the triangle inequality. Next, we approximate $\mu$ by a sequence of discrete measures $(\mu_n)_n$ which {are} below $\mu$ in the $\phc$-order.
\begin{lemma}\label{lemma:approximation}
	For every $ \mu \in \M_1(\Sp^{d-1})$  there exists a sequence $(\mu_n)_{n}$ of finitely supported measures in $\M_1(\Sp^{d-1})$ satisfying $\mu_n \phc \mu$ and $\mathrm{BarC}(\mu,\mu_n) \to 0$.
\end{lemma}

\begin{proof}
	We consider $\mu$ as a measure on $\R^d$ and assume w.l.o.g.\ that $\mu(\R^d)=1$. It is well-known (see e.g.\ \cite[Proof of Theorem~5.2]{ChGoKr23}) that there is a sequence $(\tilde{\mu}_n)_n$ of finitely supported probability measures such that  $\tilde{\mu}_n \cx \mu$ and  $\W_1(\mu,\tilde{\mu}_n) \to 0$. By Jensen's inequality, we have $\mathrm{BarC}(\mu,\tilde{\mu}_n) \le \W_1(\mu,\tilde{\mu}_n)$, so we have $\mathrm{BarC}(\mu,\tilde{\mu}_n)  \to 0$ as well. 
		
		Let  $\mu_n$ be the homogeneous marginal of $\tilde{\mu}_n$. As $\mu_n \eqph \tilde{\mu}_n$, we have $\mu_n \phc \mu$ and $\mathrm{BarC}(\mu,\mu_n) = \mathrm{BarC}(\mu,\tilde\mu_n)$. Hence, $(\mu_n)_n$ has the desired properties.
\end{proof}
}

Next, we want to show that there is attainment of the infimum in the barycentric costs if $\mathrm{BarC}(\mu,\nu)=0$. To that end, we use the following proposition, which is a consequence of Komlós' lemma  \cite{Ko67} (and also follows easily from  Mazur's lemma).  

\begin{proposition} \label{prop:komlos}  
	Let $(F_n)_n$ be a sequence in $L_1(\mu)$ satisfying $\sup_n ||F_n||_{L_1(\mu)} < \infty$. Then there are $G_n \in \textup{conv}\{ F_n, F_{n+1}, \dots \}$  and $G \in L_1(\mu)$ such that $G_n \to G$ $\mu$-a.s.
\end{proposition}


\begin{remark} \label{remark:komlos}
	We will use the following observation several times in the following  proofs: If $F_n \to F$ in $L_1(\mu)$ or pointwise, then for every choice of  $G_n \in \textup{conv}\{ F_n, F_{n+1}, \dots \}$ we have $G_n \to F$ in $L_1(\mu)$ or pointwise, respectively. 
\end{remark}

For $\nu, \rho \in \M_1(\R^d)$ we write $\rho \le \nu$ if $\nu - \rho$ is a positive measure.  $\mathcal{Q}(\mu, \le \nu)$ denotes the collection of unnormalized kernels that transport $\mu$ to a measure $\rho \le \nu$. Next, we show a compactness result for $\mathcal{Q}(\mu, \le \nu)$.

\begin{proposition}\label{prop:Qle_compact}
	Assume that $\nu$ has compact support.  Let $(q_n)_n$ be a sequence in $\mathcal{Q}(\mu,\le \nu)$. Then there are $r_n \in \mathrm{conv}\{q_n, q_{n+1},\dots \}$ and $r \in \mathcal{Q}(\mu,\le \nu)$ such that $r^x_n \to r^x$ weakly for $\mu$-a.e.\ $x$.  
\end{proposition}
\begin{proof}
	Denote $K:= \textup{supp}(\nu)$ and pick a sequence $(f_i)_{i\in \N}$ in $C(K)$ such that $0 \le f_i \le 1$ for all $i \in \N$ and that the linear span of $(f_i)_{i\in \N}$ is dense in $C(K)$ w.r.t.\ the supremum norm. Moreover, assume that {$f_0=0$ and} $f_1 =1$.
	
	For a sequence of unnormalized kernels $(r^x_n)_n$ to converge weakly for $\mu$-a.e.\ $x$ the integrals $\int f_i \, dr^x_n$  must  converge for $\mu$-a.e $x$. As a first step to construct $(r_n)_n$ we inductively construct auxiliary sequences of unnormalized kernels $(p_{k,n})_n$ for which $ \lim_n \int f_i \, dp^x_{k,n}$ exists for $\mu$-a.e.\ $x$ for all $i \leq k$. In a second step, we show that the diagonal sequence $r_n=p_{n,n}$ converges weakly $\mu$-a.s.\ to some unnormalized kernel $r$.

	\emph{Step 1: }We show by induction on $k$ that there is a family of functions $(G_k)_k$ in $L_1(\mu)$ and unnormalized kernels $p_{k,n} \in \textup{conv}\{q_j: j\ge n\}$ such that  $ \int f_i \,dp_{k,n}^x \to G_i(x)$ as $n \to \infty$ for $\mu$-a.e.\ $x$ and all $ i \leq k$.
	
	In the base case $k=0${, we set $p_{0,n} := q_n$ and have nothing to prove since $f_0=0$.}
	
	Assume the claim is true for $k$, i.e.\   there are functions $G_0, \dots, G_{k} \in L_1(\mu)$ and kernels $p_{k,n} \in \textup{conv}\{q_j: j\ge n\}$ such that $ \int f_i \,dp_{k,n} \to G_i$ $\mu$-a.s.\  as $n \to \infty$ for all $i \le k$. We write  $F_{k+1,n}(x) = \int f_{k+1} \,dp^x_{k,n}$ and note  that $||F_{k+1,n}||_{L_1(\mu)} \le \nu(\R^d)$. Applying Proposition~\ref{prop:komlos} to the sequence $(F_{k+1,n})_n$ yields a function $G_{k+1} \in L_1(\mu)$ and functions $ G_{k+1,n} \in \textup{conv}\{F_{k+1,j}: j \ge n\}$ such that $G_{k+1,n} \to G_{k+1}$ $\mu$-a.s.
	
	We can explicitly write out these convex combinations as $G_{k+1,n} = \sum_{j=n}^\infty \lambda_{n,j} F_{k+1,j}$. Note that these sums are finite, i.e.\   for given $n$ only finitely many $\lambda_{n,j}$ are non-zero. We can use these weights to define $p_{k+1,n}$ as convex combination of $\{p_{k,j} : j \ge n \}$, i.e.\   we set  $p_{k+1,n} := \sum_{j=n}^\infty \lambda_{n,j} p_{k,j}$. Considering Remark~\ref{remark:komlos} it is easy to see that $ \lim_n \int f_i \,dp_{k+1,n}^x = G_i(x)$ for $\mu$-a.e.\ $x$ and all $i \le k+1$. 
	
	\emph{Step 2:} There are kernels $r_n \in \textup{conv}\{q_j:j\ge n\}$ and a kernel $r$ such that $r_n^x \to r^x$ weakly for $\mu$-a.e.\ $x$.
	
	To see this, we pick the diagonal sequence $r_n := p_{n,n}$ and note that, again recalling Remark~\ref{remark:komlos}, $\int f_i \,dr_n \to G_i$ $\mu$-a.s. for all $i \in \N$. This implies that the sequence $(r_n^x)_n$ can have at most one weak limit point because every weak limit point $r^x$ has to satisfy $\int f_i \,dr^x = G_i(x)$ for all $i \in \N$. 
	
	As $f_1=1$, we have $r^x_n(\R^d) \to G_1(x)< \infty $ $\mu$-a.s., i.e.\  $(r_n^x)_n$ is a sequence of positive measures with bounded mass concentrated on the compact set $K$, so there exists a limit point $r^x$ by Prokhorov's theorem. 
	
	Denote $\rho := \int r^x\, \mu(dx)$. In order to show that $\rho \le \nu$, it suffices to check that $\int f_i \,d\rho \le \int f_i \,d\nu$ for all $i \in \N$. As $\mathcal{Q}(\mu, \le \nu)$ is convex and $q_n \in \mathcal{Q}(\mu,\le \nu)$ for all $n \in \N$, we have $r_n \in \mathcal{Q}(\mu,\le\nu)$ for all $n \in \N$ as well. Together with the fact that $r_n^x \to r^x$ weakly for $\mu$-a.e.\ $x$, Fatou's lemma yields for all $i \in \N$
	\[
	\int f_i \,d\rho =\iint f_i \,dr^x \mu(dx) = \int \lim_n \int f_i \,dr^x_n \mu(dx) \le \liminf_n \iint f_i \,dr^x_n \mu(dx) \le \int f_i \,d\nu. \qedhere
	\]
\end{proof}

Now we can establish {the following attainment result:}

\begin{proposition}\label{prop:BarC}
	Let $\mu, \nu \in \M_1(\R^d)$ and assume that $\nu$ has compact support. If $\mathrm{BarC}(\mu,\nu)=0$, then $\Qm(\mu,\nu) \neq \emptyset$.
\end{proposition}
\begin{proof}
	Let $q_n \in \mathcal{Q}(\mu,\nu)$ be a minimizing sequence, i.e.\  $\int |x-\ba(q_n^x)|\, \mu(dx) \to 0$. By Proposition~\ref{prop:Qle_compact} there is a sequence $(p_n)_n$ with $p_n \in \mathrm{conv}\{q_j : j\ge n\}$ and an unnormalized kernel $p \in \mathcal{Q}(\mu,\le \nu)$ such that $p_n^x \to p^x$ weakly for $\mu$-almost all $x$. 
	
	Denote  $\nu_1 := \int p^x \mu(dx)$. We show that $\ba(\nu_1)=\ba(\nu)$. First note that $\mathcal{Q}(\mu,\nu)$ is convex and $q_n \in \mathcal{Q}(\mu,\nu)$, so we have $p_n \in \mathcal{Q}(\mu,\nu)$ as well.
	
	As the function $f(x)=x$ is bounded on $\textup{supp}(\nu)$, we have $\ba(p_n^x) \to \ba(p^x)$ $\mu$-a.s. The condition that $(q_n)_n$ is a minimizing sequence means exactly that $||\ba(q^x_n)-x||_{L_1(\mu)} \to 0$. {By Remark~\ref{remark:komlos} we have $||\ba(p^x_n)-x||_{L_1(\mu)} \to 0$ as well. As the $L_1(\mu)$-limit and the $\mu$-a.s.-limit coincide if both exist, }we conclude that $\ba(p^x)=x$ $\mu$-a.s.\ and that the convergence $\ba(p_n^x) \to \ba(p^x)$ holds also in $L_1(\mu)$. Using this we find
	\[
	\int\! y \,\nu_1(dy) =\!\iint\! y \,p^x(dy) \mu(dx) = \!\int \!\lim_n\! \int\! y \,p^x_n(dy) \mu(dx) = \lim_n\! \iint\! y \,p^x_n(dy) \mu(dx) =\! \int\! y \,\nu(dy).
	\]
	Since $\nu_1 \le \nu$, the measure $\nu_2 := \nu-\nu_1$ is positive. Moreover, we have $\ba(\nu_2)=\ba(\nu)-\ba(\nu_1)=0$. We define the kernel $q$ as $q^x := p^x + \nu_2/\mu(\R^d)$. Then $q$ transports $\mu$ to $\nu_1+\nu_2=\nu$ and we have  $\ba(q^x)=\ba(p^x)+\ba(\nu_2 /\mu(\R^d))=x$. Hence, $q \in \Qm(\mu,\nu)$.
\end{proof}

Now we are ready to prove the main theorem:

\begin{proof}[Proof of Theorem \ref{thm:main}]
	By Proposition~\ref{prop:sphere} and Lemma~\ref{lemma:gluing} we can assume that all measures {involved} are concentrated on the unit sphere and hence compactly supported.
	{Suppose now} that $\mu \phc \nu$. By Proposition~\ref{prop:BarC} it suffices to show that $\mathrm{BarC}(\mu,\nu)=0$. To this end, we choose an approximating sequence $(\mu_n)_n$ for $\mu$ as in Lemma~\ref{lemma:approximation}. Using the triangle inequality for barycentric costs~\eqref{eq:gluingineq}, we find 
	\[
	\mathrm{BarC}(\mu,\nu) \leq \mathrm{BarC}(\mu,\mu_n) + \mathrm{BarC}(\mu_n,\nu).
	\]
	{The first term tends to zero as $n \to \infty$ by Lemma~\ref{lemma:approximation} and} the second term {vanishes} by Proposition~\ref{prop:mu_endlich} since $\mu_n \phc \mu \phc \nu$.
	
	For the reverse direction, {suppose that $\Qm(\mu,\nu) \neq \emptyset$. The same calculation as in \eqref{eq:jensen} shows that $\int f \, d\mu \le \int f \, d\nu$ for every support function $f : \R^d \to \R$.} 
\end{proof}

\section{Equivalence of Theorem~\ref{thm:strassen} and Theorem~\ref{thm:Strassen_Gozlan}(b)}
{The aim of this section is to show that the classical Strassen theorem in $\R^d$ is equivalent to the Strassen-type theorem for $\phc$ in the special case\footnote{%
		{This is precisely the setting of Theorem~\ref{thm:Strassen_Gozlan}(b). Let $\mu \phc \nu \in \mathcal{M}_1(\R^{d+1})$ such that the convex hull of $\mathrm{supp}(\nu)$ is compact and does not contain 0. After possibly rotating the coordinate system, we find that  $\mathrm{co}(\mathrm{supp}(\nu)) \subset (\R^d \times (0,\infty))\cup \{ 0 \}$. By Remark~\ref{rem:cospt}, the same is true for $\mu$. As removing a point mass in 0 preserves the $\eqph$-class, we can assume w.l.o.g.\ that $\mu$ is concentrated on $\R^d \times (0,\infty)$.}
	} that the measures are concentrated on $\R^d \times (0,\infty) \subset \R^{d+1}$. To this end, we identify $\R^d$ with the affine subspace  $\R^d \times \{1\} \subset \R^{d+1}$ using the embedding $\iota: \R^d \to \R^{d+1} : x \mapsto (x,1)$. The notion of perspective function (see e.g.\ \cite[Section~2]{Co18})  provides a one-to-one correspondence between convex functions on $\R^d$ and support functions on $\R^{d} \times (0,\infty)$.

\begin{definition}
	The perspective of a convex function $f: \R^d \to \R$ is the l.s.c.\ support function
	$\widehat{f}: \R^{d+1}\to \R \cup \{+\infty\}$ defined by

	\[
	\widehat{f}(x,\eta)=
	\begin{cases} 
		\eta f(x/\eta) &  \eta > 0,\\
		\lim\limits_{\alpha \to 0^+} \alpha f(x/\alpha ) & \eta=0, \\
		+\infty & \eta < 0.
	\end{cases} 
	\]
\end{definition}
We have $\widehat{f} \circ \iota = f$ for every convex function $f : \R^d \to \R$.  Hence, for $\mu,\nu \in \prob(\R^d)$ 
\begin{align}\label{eq:cxphc1}
	\mu \cx \nu \quad \text{if and only if} \quad \iota_\#\mu \phc \iota_\#\nu.
\end{align}

\begin{proof}[Proof that Theorem~\ref{thm:Strassen_Gozlan}(b) implies Theorem~\ref{thm:strassen}]
	Let $\mu \cx \nu \in \prob(\R^d)$. By \eqref{eq:cxphc1}, we have $\iota_\#\mu \phc \iota_\#\nu$, so by Theorem~\ref{thm:Strassen_Gozlan}(b) there is a kernel $q\in \Qm(\iota_\#\mu , \iota_\#\nu)$. Let $\pr_{\R^d} : \R^{d+1} \to \R^d$ be the projection onto the first $d$ components. Then the kernel defined by $p^x:=\pr_{\R^d\#}q^{\iota(x)}$ is in $\Qm(\mu,\nu)$. To conclude that $p$ is a martingale kernel, it suffices to show that  $p^x(\R^{d})=1$ for $\mu$-a.e.\ $x$. As $q^{\iota (x)}$ is concentrated on $\R^d \times \{1\}$, we have  
	\[	
	p^x(\R^{d}) = q^{\iota(x)}(\R^{d+1}) = \int 1 \, q^{\iota(x)}(dy) = \int y_{d+1} \, q^{\iota(x)}(dy) =1 . \qedhere
	\]
\end{proof}
If $\mu \in \mathcal{M}_1(\R^{d+1})$ is concentrated on $\R^d \times (0,\infty)$, there is a representative of its $\eqph$-class that is supported on $\R^d \times \{1\}$. In order to construct such an element, set
\[
\widehat{\mu}:=\phi_\#(x_{d+1}\cdot\mu),
\]
where $ \phi(x_1, \dots, x_d, x_{d+1}) = (x_1/x_{d+1},\dots,x_d/x_{d+1}).
$
We then have $\int f \, d\widehat{\mu} = \int \widehat{f} \, d\mu$ for every convex function  $f : \R^d \to \R$. For every support function $f : \R^{d+1} \to \R$ we have $\widehat{f \! \circ \! \iota } =f $ on $\R^d \times (0,\infty)$ and hence 
\[
\int f \, d\iota_\#\widehat{\mu} =\int f\circ \iota \, d\widehat{\mu} = \int \widehat{f \! \circ \! \iota } \, d\mu= \int f \, d\mu.
\]
Therefore, we have  $\mu \eqph \iota_\#\widehat{\mu}$ by Corollary~\ref{cor:ph=phc}.

\begin{proof}[Proof that Theorem~\ref{thm:strassen} implies Theorem~\ref{thm:Strassen_Gozlan}(b)]
	Let $\mu \phc \nu \in \M_1(\R^{d+1}) $ both be concentrated on $\R^d \times  (0,\infty)$. Then $\int x_{d+1} \,\mu(dx) = \int x_{d+1}\, \nu(dx)$ and, by scaling $\mu$ and $\nu$, we can assume these integrals to be 1. In particular, we have $\widehat{\mu}, \widehat{\nu} \in \prob(\R^{d})$. 
	
	As $\mu \eqph \iota_\#\widehat{\mu}$ and  $\nu \eqph \iota_\#\widehat{\nu}$, we have $\iota_\#\widehat{\mu} \phc \iota_\#\widehat{\nu}$, moreover, Proposition~\ref{prop:char_ph_equiv} implies that $\Qm(\mu,\iota_\#\widehat{\mu}) \neq \emptyset$ and $\Qm(\iota_\#\widehat{\nu},\nu) \neq \emptyset$. So, using the gluing lemma, it suffices to find a moment-preserving kernel between $\iota_\#\widehat{\mu}$ and $\iota_\#\widehat{\nu}$.  
	
	By \eqref{eq:cxphc1} we have $\widehat{\mu} \cx \widehat{\nu}$, so by Theorem~\ref{thm:strassen} there exists a martingale kernel $p$ between $\mu$ and $\nu$.  Clearly, the kernel defined by $ q^{\iota(x)} :=\iota_\#p^x$ is in  $\Qm(\iota_\#\widehat{\mu},\iota_\#\widehat{\nu})$.
\end{proof}%
}





\medskip
\noindent
{\bf Acknowledgment:}
This research was funded in whole or in part by by the Austrian Science
Fund (FWF) through projects 10.55776/P35197 and 10.55776/P34743. For open access purposes, the author has applied a CC BY public copyright license to any author accepted manuscript version arising from this submission.	
	
Both authors thank the anonymous referees for their suggestions and remarks that significantly improved the presentation of this note and Mathias Beiglböck, Leo Brauner, and Gudmund Pammer for many helpful discussions and remarks.


\end{document}